\documentclass[american,a4paper]{amsart}
\pdfoutput=1
\usepackage[T1]{fontenc}
\usepackage[utf8]{inputenc}
\usepackage{lmodern}
\usepackage{mathrsfs}
\usepackage{enumitem}
\usepackage{mathtools}
\usepackage{amssymb}
\usepackage{bm}
\usepackage[unicode=true,pdfusetitle,
 bookmarks=true,bookmarksnumbered=false,bookmarksopen=false,
 breaklinks=false,pdfborder={0 0 1},backref=false,colorlinks=false]{hyperref}
\usepackage{breakurl}
\usepackage[foot]{amsaddr}
\usepackage{cite}

\newtheorem{thm}{Theorem}[section]
\newtheorem{lem}[thm]{Lemma}

\newtheorem{cor}[thm]{Corollary}

\theoremstyle{definition}

\theoremstyle{remark}
\newtheorem{rem}{Remark}[section]

\newcommand\vect[1]{{\mathbf{#1}}}
\DeclareMathOperator{\vspan}{span}
\DeclareMathOperator{\diag}{diag}
\DeclareMathOperator{\Real}{Re}

\newcommand\dd{\textup{d}}
\newcommand\misvar{\,\cdot\,}
\title[A Liouville-type result for non-cooperative Fisher--KPP systems]{A Liouville-type result for non-cooperative Fisher--KPP systems and nonlocal equations in cylinders}
\author{L\'{e}o Girardin}
\thanks{This work was supported by a public grant as part of the Investissement d'avenir project,
reference ANR-11-LABX-0056-LMH, LabEx LMH. This work has been carried out in the framework of 
the NONLOCAL project (ANR-14-CE25-0013) funded by the French National Research Agency (ANR)}
\address[L. G.]{Universit\'{e} Paris-Saclay, CNRS, Laboratoire de Math\'{e}matiques d'Orsay, 91405 Orsay Cedex, France}
\email{leo.girardin@math.u-psud.fr}

\author{Quentin Griette}
\address[Q. G.]{Universit\'{e} de Bordeaux, Institut de Math\'{e}matiques de Bordeaux, B\^{a}t. A33 - Bureau 211, 351, cours de la Lib\'{e}ration, 33405 Talence, France}
\email{quentin.griette@math.u-bordeaux.fr}

\begin{document}
\begin{abstract}
    We address the uniqueness of the nonzero stationary state for a reaction--diffusion system
    of Fisher--KPP type that does not satisfy the comparison principle. Although the uniqueness is false in general,
    it turns out to be true under biologically natural assumptions on the parameters. This Liouville-type result is then used
    to characterize the wake of traveling waves. All results are extended to an analogous
    nonlocal reaction--diffusion equation that contains as a particular case the cane toads equation with bounded traits.
\end{abstract}

\maketitle

\section{Introduction}

We investigate the reaction--diffusion system
\begin{equation}\label{eq:RD_system}
	\partial_t\vect{u}-\mathbf{D}\partial_{xx}\vect{u}= \vect{M}\vect{u}+\vect{u} - \mathbf{u}\circ(\mathbf{C}\mathbf{u}),
\end{equation}
where $t\in\mathbb R$ is a time variable, $x\in\mathbb R$ is a space variable, 
$\vect{u}(t,x)$ is a nonnegative column vector
\footnote{In the whole paper, nonnegativity and positivity of vectors and matrices are understood component-wise.}
collecting $N\geq 2$ phenotype densities among a species, $\mathbf{D}$ is a diagonal matrix 
collecting positive diffusion rates, $\circ$ is the Hadamard product (component-by-component product)
between two vectors and $\mathbf{M}$ and $\mathbf{C}$ are square matrices collecting respectively mutation rates and competition rates
and satisfying the following standing assumptions (below and in the whole paper, $\mathbf{1}=(1,1,\dots,1)^\textup{T}\in\mathbb{R}^N$).

\begin{enumerate}[label=$({A}_{\arabic*})$]
	\item \label{hyp:line-sum-symmetry}
	    The matrix $\mathbf{M}\in\mathscr{M}_{N,N}(\mathbb{R})$ is essentially
	    nonnegative (namely, with nonnegative off-diagonal coefficients), irreducible, line-sum-symmetric
	    (namely, $\mathbf{M}\mathbf{1}=\mathbf{M}^{\textup{T}}\mathbf{1}$)
	    and admits $(0,\mathbf{1})$ as Perron--Frobenius eigenpair (namely, $\mathbf{M}\mathbf{1}=\mathbf{0}$).
	\item \label{hyp:normal}
		The matrix $\mathbf{C}\in\mathscr{M}_{N,N}(\mathbb{R})$ is positive, normal and admits
		$(1,\mathbf{1})$ as Perron-Frobenius eigenpair (namely, $\mathbf{C}\mathbf{1}=\mathbf{1}$). We denote 
		$\mathbf{U}\in\mathscr{M}_{N,N}(\mathbb{C})$ the unitary matrix such that 
		$\mathbf{U}\mathbf{C}\mathbf{U}^{-1}=\mathbf{U}\mathbf{C}\overline{\mathbf{U}}^{\textup{T}}$ is diagonal. 
	\item \label{hyp:stable}
		The spectrum of $\mathbf{C}$ is contained in the complex closed right-half plane.
\end{enumerate}

We are interested more specifically in the associated traveling wave equation
\begin{equation}\label{eq:TW}
	-\mathbf{D}\vect{p}''-c\vect{p}'=\vect{M}\vect{p}+\vect{p}-\mathbf{p}\circ(\mathbf{C}\vect{p}),
\end{equation}
satisfied by solutions of the system \eqref{eq:RD_system} of the form $\mathbf{u}:(t,x)\mapsto\mathbf{p}(x-ct)$.
This equation might be supplemented with asymptotic conditions for the profile $\mathbf{p}$. 
The asymptotic conditions of classical traveling waves $(\mathbf{p},c)$ \cite{Girardin_2016_2} are
\begin{equation}\label{eq:asymptotic_conditions}
    \lim_{+\infty}\mathbf{p}=\mathbf{0},\quad\min_{i\in[N]}\liminf_{-\infty}p_i\geq0,\quad\max_{i\in[N]}\liminf_{-\infty}p_i>0,
\end{equation}
where $[N]$ denotes (here and in the rest of the paper) the set $\left\{ 1,2,\dots,N \right\}$.

By \ref{hyp:line-sum-symmetry} and \ref{hyp:normal}, $\mathbf{1}$ is a constant steady state of the system \eqref{eq:RD_system}.

\subsection{Main results}

Our main result is the following theorem.
\begin{thm}[Liouville-type result]\label{thm:main}
    Assume \ref{hyp:line-sum-symmetry}, \ref{hyp:normal} and \ref{hyp:stable}. Then, for any $c\in\mathbb{R}$, $\mathbf{1}$ is the unique bounded solution $\mathbf{p}$ of \eqref{eq:TW} such that
    $\min_{i\in[N]}\inf_{\mathbb{R}} p_i>0$.
\end{thm}

The main consequences of this theorem are the two following corollaries, deduced from standard elliptic
estimates and limiting procedures \cite{Gilbarg_Trudin} as well as a strong positivity property \cite[Theorem 1.1]{Girardin_2016_2}.
\begin{cor}[Uniqueness of the nonzero steady state]\label{cor:unique}
	Assume \ref{hyp:line-sum-symmetry}, \ref{hyp:normal} and \ref{hyp:stable}. Then $\mathbf{1}$ is the unique bounded nonnegative nonzero stationary solution of
	\eqref{eq:RD_system}, namely the unique bounded nonnegative nonzero solution $\mathbf{p}$ of \eqref{eq:TW} with $c=0$.
\end{cor}

\begin{cor}[Limit behavior of the traveling waves]\label{cor:limit_TW}
    Assume \ref{hyp:line-sum-symmetry}, \ref{hyp:normal} and \ref{hyp:stable}. Then all 
    bounded solutions $(\mathbf{p},c)$ of \eqref{eq:TW}-\eqref{eq:asymptotic_conditions}
    actually satisfy $\lim_{-\infty}\mathbf{p}=\mathbf{1}$.
\end{cor}

\subsection{Extension to nonlocal equations}

Those results extend to continuous limits $N\to+\infty$, provided the limit equation has a similar structure. Below we illustrate this principle by focusing on an equation supplemented with Neumann boundary conditions, though it would also be possible to adapt our arguments in the periodic framework with no additional difficulty. 

We consider
\begin{equation}\label{eq:TW_continuous}
	-d(y)\partial_{\xi\xi}p-c\partial_{\xi}p=\nabla_y\cdot(\sigma(y)\nabla_y p)
	+M[p(\xi)](y)+p(\xi,y)\left( 1-K[p(\xi)](y) \right)
\end{equation}
set on $(\xi, y)\in \mathbb R\times\Omega$ for a smooth domain $\Omega\subset \mathbb R^Q$ ($Q\geq 1$ and $\partial \Omega $ is $\mathscr C^2$) and supplemented with homogeneous Neumann boundary conditions at $y\in\partial\Omega$.
Above, $d\in\mathscr{C}\left(\overline{\Omega},(0,+\infty)\right)$, 
$\sigma\in\mathscr{C}^1\left( \overline{\Omega},(0,+\infty) \right)$,
\begin{equation*}
    M[p(\xi)]=\int_{\Omega} m(\misvar,\tilde{y})(p(\xi,\tilde{y})-p(\xi,\misvar))\textup{d}\tilde{y},
    \quad
    K[p(\xi)]=\int_{\Omega} k(\misvar,\tilde{y})p(\xi,\tilde{y})\textup{d}\tilde{y},   
\end{equation*}
for some $m,k\in\mathscr{C}(\Omega^2,(0,+\infty))$. Defining naturally the adjoint operators $M^\star$ and $K^\star$,
the assumptions \ref{hyp:line-sum-symmetry}, \ref{hyp:normal} and \ref{hyp:stable} extend to the continuous equation as follows:
\begin{enumerate}[label=$({A}'_{\arabic*})$]
	\item \label{hyp:cont_line-sum-symmetry}
	    The function $\sigma(y)\in \mathscr C^1\left(\overline{\Omega}\right)$ is positive and the function 
	    $m\in \mathscr C(\Omega^2)$ is nonnegative, bounded and satisfies 
	    $\int_\Omega m(\misvar,z)\textup{d}z=\int_\Omega m(z,\misvar)\textup{d}z$.
	\item \label{hyp:cont_normal}
		The function $k\in \mathscr C(\overline{\Omega}^2)$ is positive and the induced operator $K[p]=\int_{\Omega} k(\misvar, z)p(z)\dd z$ acting on the Hilbert space  $ L^2(\Omega) $ is normal. Moreover, the constant function $y\in\Omega\mapsto 1$ is an eigenvector of $K$ associated with the eigenvalue $1$ (namely, $K[1]=1$).
	\item \label{hyp:cont_stable}
		The spectrum of $K$ (considered as an operator acting on  $L^2(\Omega)$) is contained in the complex closed right-half plane.
\end{enumerate}
The continuous version of Theorem \ref{thm:main} reads as follows.
\begin{thm}\label{thm:main_continuous}
    Assume \ref{hyp:cont_line-sum-symmetry}, \ref{hyp:cont_normal} and \ref{hyp:cont_stable}. Then, for any $c\in\mathbb{R}$, $1$ is the unique bounded solution $p$ of \eqref{eq:TW_continuous}
    such that $\inf_{\mathbb{R}\times\Omega}p>0$.
\end{thm}
We deduce just as before the uniqueness of the stationary states and the uniform convergence to the unique stationary state 
in the wake of the waves for \eqref{eq:TW_continuous}, provided a uniform estimate from below can be shown.
\begin{cor}\label{cor:unique_cont}
    Assume \ref{hyp:cont_line-sum-symmetry}, \ref{hyp:cont_normal} and \ref{hyp:cont_stable}. Then $1$ is the unique bounded solution of \eqref{eq:TW_continuous} with positive infimum in $\mathbb{R}\times\Omega$ and with $c=0$.
\end{cor}
\begin{cor}\label{cor:TW_limit_cont}
    Assume \ref{hyp:cont_line-sum-symmetry}, \ref{hyp:cont_normal} and \ref{hyp:cont_stable}. 
    Then all bounded classical solutions $(p,c)$ of \eqref{eq:TW_continuous} such that 
    \[
	\lim_{\xi\to+\infty}\sup_{y\in\Omega}p(\xi,y)=0\text{ and }\liminf_{\xi\to-\infty}\inf_{y\in\Omega}p(\xi,y)>0
    \]
    actually satisfy 
    \[
	\lim_{\xi\to-\infty}\sup_{y\in\Omega}|p(\xi,y)-1|=0.
    \]
\end{cor}

\subsection{Organization of the paper}
In Section 2, we discuss the assumptions, the results and the literature. In Section 3, we prove Theorem \ref{thm:main}. In
Section 4, we prove Theorem \ref{thm:main_continuous}.

\section{Discussion}


\subsection{The conditions on $\mathbf{M}$} 

By definition, a matrix is line-sum-symmetric if the sum of coefficients in each of its rows equals the sum of coefficients
in the corresponding column. Symmetric matrices and circulant matrices are line-sum-symmetric. Although the notions of 
symmetry, circulancy and line-sum-symmetry coincide in dimension $2$, in dimension $3$ and higher, there are 
line-sum-symmetric matrices that are neither symmetric nor circulant, as shown by the following counter-example:
\begin{equation*}
    \begin{pmatrix}
	a & 2b & 0 \\
	b & c & b \\
	b & 0 & d
    \end{pmatrix}
    \quad\text{with }a,b,c,d\in\mathbb{R}.
\end{equation*}

The study of line-sum-symmetric matrices was initiated by Eaves, Hoffman, Rothblum and Schneider \cite{Eaves_1985}. 
Roughly speaking, these matrices conveniently
generalize symmetric matrices when what we have in mind is summation of lines or rows of linear systems 
\cite[Corollary 3]{Eaves_1985}, which is the case in this paper and more generally whenever we want to ``integrate by parts'' 
in a discrete variable. As such, they 
recently appeared in the literature on reaction--diffusion systems \cite{Cantrell_Cosner_Lou_2012,Cantrell_Cosner_Lou_Ryan_2012}.

\subsubsection{The symmetric case}
In the symmetric case, which arises in many applications, our assumption \ref{hyp:line-sum-symmetry}
on $\mathbf{M}$ comes down to assuming that $\mathbf M$ has an ``integration by parts'' formula:
\begin{equation*}
    \langle \mathbf{M}\mathbf{u},\mathbf{v}\rangle 
    =-\frac12\sum_{i,j\in[N]} m_{i,j}\left(u_i-u_j)(v_i-v_j\right).
\end{equation*}
where $\langle\cdot,\cdot\rangle$ is the canonical (Hermitian) scalar product on $\mathbb{C}^N$.
A particularly natural example is the explicit Euler scheme for the one-dimensional heat equation
with periodic boundary conditions:
$\mathbf{M}=-\boldsymbol{\nabla_{\textup{D}}}^{\textup{T}}\mathbf{\Sigma}\boldsymbol{\nabla_{\textup{D}}}$,
$\mathbf\Sigma=\diag(\sigma_1, \sigma_2, \ldots, \sigma_N)$ ($\sigma_i>0$) and 
\begin{equation*}
	\boldsymbol{\nabla_{\textup{D}}}=\begin{pmatrix}
		-1 & 0 & 0 & \cdots &  &  & 1 \\
		1 & -1 & 0& 0 & \cdots & \cdots  & 0 \\
		0 & 1 & -1 & 0 & 0 & \cdots& \vdots \\
		\vdots & \vdots & \vdots & \vdots & \vdots & &\vdots \\
		0 & 0 & \cdots &0 & & 1 & -1
	\end{pmatrix}.
\end{equation*}
The expanded form of $\mathbf{M}$ is
\begin{equation*}
	\begin{pmatrix}
		-\sigma_1-\sigma_2 & \sigma_2 & 0 & \dots & 0 & \sigma_1 \\
		\sigma_2 & -\sigma_2-\sigma_3 & \sigma_3 & 0 & \dots & 0 \\
		0 & \sigma_3 & -\sigma_3-\sigma_4 & \sigma_4 & 0 & \dots \\
		\vdots & \vdots & \vdots & \vdots & \vdots & \vdots \\
		\sigma_1 & 0 & \dots & 0 & \sigma_N & -\sigma_N-\sigma_1
	\end{pmatrix}
	\quad\text{if }N\geq 3,
\end{equation*}
\begin{equation*}
	(\sigma_1+\sigma_2)
	\begin{pmatrix}
		-1 & 1 \\
		1 & -1
	\end{pmatrix}
	\quad\text{if }N=2.
\end{equation*}

Neumann boundary conditions can be obtained by replacing the first line in $\boldsymbol{\nabla_{\textup{D}}}$ by zero
and also satisfy \ref{hyp:line-sum-symmetry}.
On the contrary, Dirichlet boundary conditions are qualitatively different (in particular,
$\mathbf{1}$ cannot be a solution anymore) and are therefore outside the scope of this paper.
 Note that non-tridiagonal matrices can also be obtained in the form $-\boldsymbol{\nabla}_\textup{D}^\textup{T}\mathbf{\Sigma}\boldsymbol{\nabla}_\textup{D}$ by allowing $\boldsymbol{\nabla}_\textup{D}$ to be non-square: as an example,  discretization of divergence-form operators in two-dimensional domain such as
\begin{equation*}
	\mathbf{M}=\begin{pmatrix} 
		-\sigma_1-\sigma_5& 0& \sigma_1& 0& \sigma_5 \\
		0 & -\sigma_2-\sigma_6 & \sigma_2 & \sigma_6& 0 \\
		\sigma_1 & \sigma_2 & -\sigma_1-\sigma_2-\sigma_3-\sigma_4 & \sigma_3 & \sigma_4 \\
		0 & \sigma_6 & \sigma_3 & -\sigma_3-\sigma_6 & 0 \\
		\sigma_5 & 0 & \sigma_4 & 0 & -\sigma_4 -\sigma_5
	\end{pmatrix}
\end{equation*}
are not always tridiagonal.
In this case $\boldsymbol{\nabla_{\textup{D}}}\in\mathscr{M}_{10, 5}(\mathbb R)$ corresponds to a discrete gradient operator on a cell with four boundary points and one interior point, and $\mathbf{\Sigma}\in\mathscr{M}_{10,10}(\mathbb R)$ encodes the diffusion rates. 

In addition to the divergence-form differential part presented above, $\mathbf{M}$ might also contain the discretization
of a nonlocal integral operator, as hinted by \eqref{eq:TW_continuous}.

\subsection{The conditions on $\mathbf{C}$}

The assumption that the Perron--Frobenius eigenvalue of $\mathbf{C}$ is unitary ($\lambda_{\textup{PF}}(\mathbf{C})=1$)
is done without loss of generality (up to replacing $(\mathbf{p},\mathbf{C})$ by 
$(\lambda_{\textup{PF}}(\mathbf{C})\mathbf{p},\lambda_{\textup{PF}}(\mathbf{C})^{-1}\mathbf{C})$).
However the assumption that $\mathbf{1}$ is a Perron--Frobenius eigenvector is a true assumption, not 
satisfied in general.

The set of real positive normal matrices contains as particular subsets the set of real positive
circulant matrices and the set of real positive symmetric matrices (skew-symmetric
and orthogonal matrices are normal but cannot be positive).
The following counter-example shows that there are matrices satisfying \ref{hyp:normal} and \ref{hyp:stable} that are neither symmetric nor circulant:
\begin{equation*}
    \begin{pmatrix}
	a & b & c & d \\
	b & a & d & c \\
	d & c & a & b \\
	c & d & b & a
    \end{pmatrix}
    \quad\text{with }a,b,c,d>0,\ a+b+c+d=1.
\end{equation*}
(The eigenvalues of this matrix are $1$, $a+b-c-d$, $a-b\pm\textup{i}|c-d|$ and therefore \ref{hyp:stable} is 
satisfied as soon as $a\geq b$ and $a+b\geq c+d$.)

In fact, a polynomial in any permutation matrix is normal. It is therefore possible to construct such counterexamples in any dimension $N\geq 4$, by selecting a permutation matrix associated with a cycle of maximal length which is not a power of the circular permutation.

\subsubsection{The circulant case}
In the circulant case, which is of particular interest to us, there exists a positive vector 
$\bm{\phi}\in\mathbb{R}^N$ such that the matrix $\mathbf{C}$ is written as 
$\mathbf{C}=\left( \phi_{i-j} \right)_{i,j\in[N]}$, $\bm{\phi}$ being periodically extended by
\begin{equation*}
    \phi_{i-j}=\begin{cases} 
    	\phi_{i-j}, & \text{if } i-j\geq 1, \\
    	\phi_{N+i-j}, & \text{if } i-j\leq 0.
    \end{cases}
\end{equation*}
The expanded form of $\mathbf{C}$ is then
\begin{equation*}
    \begin{pmatrix}
        \phi_N & \phi_{N-1} & \dots & \phi_{1} \\
        \phi_{1} & \phi_N & \dots & \phi_{2} \\
        \vdots & \vdots & \vdots & \vdots \\
        \phi_{N-1} & \phi_{N-2} & \dots & \phi_N
    \end{pmatrix}
\end{equation*}
and the product $\mathbf{C}\mathbf{u}$ can be rewritten as $\bm{\phi}\star \mathbf{u}$,
where $\star$ is the discrete circular convolution operator:
\begin{equation*}
	(\bm{\phi}\star \vect{u})_i=\sum_{j=1}^{N}\phi_{i-j}u_j.
\end{equation*}

Defining the normalized discrete Fourier transform matrix as
\begin{equation*}
    \mathbf{U}_{\textup{DFT}}=\frac{1}{\sqrt{N}}\left( \exp\left( -\frac{2\textup{i}\pi}{N}(j-1)(k-1) \right) \right)_{j,k\in[N]},
\end{equation*}
we find that $\mathbf{U}_{\textup{DFT}}=\mathbf{U}$ for any circulant matrix $\mathbf{C}$.
In particular, $\mathbf{1}$ is automatically a Perron--Frobenius eigenvector (and 
the normalization $\lambda_{\textup{PF}}(\mathbf{C})=1$ reads $\sum_{i=1}^N \phi_i=1$).
Moreover, the following equalities hold true:
\begin{equation*}
    \mathbf{U}\mathbf{C}\mathbf{U}^{-1}\mathbf{U}\mathbf{u}=\mathbf{U}\mathbf{C}\mathbf{u}
    =\mathbf{U}(\bm{\phi}\star\mathbf{u})=\sqrt{N}(\mathbf{U}\bm{\phi})\circ(\mathbf{U}\mathbf{u}).
\end{equation*}
It follows easily that the spectrum of $\mathbf{C}$ is contained in the closed right-half plane
if and only if $\mathbf{U}\bm{\phi}$, namely the discrete Fourier transform of $\bm{\phi}$,
is valued in the closed right-half plane.

Last, we point out additional alternative writings of the reaction term:
\begin{equation*}
    \mathbf{u}-\left( \mathbf{C}\mathbf{u} \right)\circ\mathbf{u}=
\vect{u}\circ\left(\mathbf{1} - \bm{\phi}\star\mathbf{u}\right)=
-\vect{u}\circ\left(\bm{\phi}\star\left(\mathbf{u}-\mathbf{1}\right)\right),
\end{equation*}

\subsection{The case $N=2$} 
\label{sec:N=2}

In the case $N=2$, the matrices $\mathbf{M}$ and $\mathbf{C}$ can be rewritten as depending on two parameters only:
\begin{align*}
	\mathbf{M}&=
	\begin{pmatrix} 
		-\sigma & \sigma \\ 
		\sigma & -\sigma 
	\end{pmatrix}, & \mathbf{C}=
	\begin{pmatrix}
		1-\gamma & \gamma \\
		\gamma & 1-\gamma
	\end{pmatrix}, 
\end{align*}
where $\sigma>0$ and $\gamma\in (0, 1)$. The linear stability of $\mathbf{1}$ can be decided by computing the eigenvalues $\lambda_{\pm}^{\mathbf{M}-\mathbf{C}}$ of the matrix $\mathbf{M}-\mathbf{C}$,
\begin{equation*}
	\lambda_{\pm}^{\mathbf{M}-\mathbf{C}} = -1-(\gamma-\sigma) \pm |\gamma-\sigma|,
\end{equation*}
while the eigenvalues of $\mathbf C$ are $\lambda_1^{\mathbf C}=1$ and $\lambda_-^{\mathbf{C}}= 1-2\gamma$.
Therefore, $\mathbf{M}-\mathbf{C}$ has always one negative eigenvalue $\lambda_-^{\mathbf{M}-\mathbf{C}}<0$ and the behavior of $\lambda_+^{\mathbf{M}-\mathbf{C}}$ depends on the value of $\gamma$: 
\begin{enumerate}[label={\alph*)}]
	\item if $\gamma\in (0,1/2)$ (in which case \ref{hyp:stable} holds), $\lambda_1^{\mathbf{M}-\mathbf{C}}$ always stays negative,
	\item if $\gamma\in (1/2, 1)$ (in which case \ref{hyp:stable} does not hold),
		\begin{align*}
			\lambda_+^{\mathbf{M}-\mathbf{C}}&>0 \text{ if }0<\sigma<\sigma^*:=\gamma-\frac{1}{2}, \\
			\lambda_+^{\mathbf{M}-\mathbf{C}}&<0 \text{ if }\sigma>\sigma^*.
		\end{align*}
\end{enumerate}
In the latter case, using $\sigma $ as a bifurcation parameter, a local  bifurcation is occurring when decreasing $\sigma $ below $\sigma^*$ and  two stable equilibria emerge when $\mathbf{1} $ loses stability. In particular, in this case there are solutions to \eqref{eq:TW} other than the constant $\mathbf{1}$ which are bounded from below. This is confirmed by the result in \cite[Proposition 3.4]{Cantrell_Cosner_Yu_2018}. 

\subsection{KPP systems}
The system \eqref{eq:RD_system} is a particular example of KPP systems. The first author studied
these systems in \cite{Girardin_2016_2,Girardin_2017,Girardin_2018,Girardin_2016_2_add}. 
The second author studied them with collaborators in
\cite{Griette_Raoul,Alfaro_Griette} and gave an epidemiological interpretation in \cite{Griette_Raoul_}.
Other important mathematical references are 
\cite{Dockery_1998,Barles_Evans_S,Morris_Borger_Crooks,Cantrell_Cosner_Yu_2018,Cantrell_Cosner_Yu_2019}. For a detailed
overview of the literature, we refer to \cite{Girardin_2016_2}.

These nonlinear, non-variational and non-cooperative (or non-monotone, in other words devoid 
of comparison principle) reaction--diffusion systems are referred to as ``KPP systems'' 
due to their structural similarity with the scalar Fisher--KPP equation,
\begin{equation*}
    \partial_t u - \partial_{xx} u = u(1-u).
\end{equation*}
(This scalar equation can actually be understood as a KPP system of dimension $1$.)
This similarity mainly concerns the behavior close to $\mathbf{u}=\mathbf{0}$ and it leads to several
classical results: a sharp persistence--extinction
criterion \cite{Girardin_2016_2,Girardin_2016_2_add}, the existence of traveling waves for all speeds larger than 
or equal to a linearly determined 
minimal wave speed $c^\star$ \cite{Girardin_2016_2,Griette_Raoul,Morris_Borger_Crooks}, the equality between this
minimal wave speed and the asymptotic speed of spreading for initially compactly supported solutions of the Cauchy problem 
\cite{Barles_Evans_S,Girardin_2016_2} and an exponential equivalent of the profile at the leading edge 
\cite{Girardin_2017,Morris_Borger_Crooks}.

However, away from $\mathbf{u}=\mathbf{0}$ and in particular in the wake of a traveling wave solution $\mathbf{p}\left( x-ct \right)$, 
the picture is more complicated. For two-component systems, locally uniform convergence of the solutions of the Cauchy problem 
to a unique constant steady state can be proved in many cases (and directly implies the convergence in the wake of the traveling waves)
\cite[Appendix B]{Girardin_2018}, \cite{Griette_Raoul}, but bistable cases (corresponding to strongly 
competitive systems with weak mutations) still exist \cite{Cantrell_Cosner_Yu_2018,Girardin_2017} 
and remain elusive -- in particular, traveling
waves connecting $\mathbf{0}$ to an unstable constant steady state exist in some particular bistable cases \cite{Girardin_2017}.
For systems of any size but where $\mathbf{D}=\mathbf{I}$ and $\mathbf{C}=\mathbf{1}\mathbf{a}^\textup{T}$, locally uniform convergence
of the solutions of the Cauchy problem to a unique constant steady state can be established \cite{Girardin_2017}, 
but these assumptions are in fact so strong that
the system is basically reduced to a scalar Fisher--KPP equation projected along the Perron--Frobenius eigenvector of the 
linear part of the reaction term. More recent results confirm that, as soon as there are at least three components,
convergence fails in general. In particular, for circulant matrices $\mathbf{M}$ and $\mathbf{C}$, Hopf bifurcations can occur and
these typically lead to the formation of limit cycles, periodic wave trains, pulsating traveling waves and propagating terraces 
\cite{Girardin_2018}.

In this regard, the main result of this paper provides some sufficient conditions to prevent the formation of these oscillations 
in the elliptic and traveling wave problems. In the class of pairs $(\mathbf{M},\mathbf{C})$ satisfying \ref{hyp:line-sum-symmetry} and 
\ref{hyp:normal}, the sharpness of \ref{hyp:stable} (the spectrum of $\mathbf{C}$ is in the right-half plane) 
can be discussed as follows:
\begin{itemize}
	\item in view of the Hopf-bifurcating case in \cite{Girardin_2018}, the system can be
		oscillatory if $\mathbf{C}$ admits an eigenvalue with negative real part and 
		nonzero imaginary part;
	\item in view of the two-component case discussed in Section \ref{sec:N=2} (see also \cite[Proposition 3.4]{Cantrell_Cosner_Yu_2018}), there can be a multiplicity of positive constant equilibria when at least one eigenvalue of $\mathbf{C}$ is real and negative.
\end{itemize}
In the class of pairs $(\mathbf{M},\mathbf{C})$ satisfying \ref{hyp:normal}, \ref{hyp:stable} and the mere irreducibility of $\mathbf{M}$,
the sharpness of \ref{hyp:line-sum-symmetry} is unclear. The proof presented here heavily relies on the line-sum-symmetry of 
$\mathbf{M}$ and cannot be extended to more general matrices $\mathbf{M}$ (see Remark \ref{rem:sharpness_of_A1} below).

Let us point out that the convergence result here is strikingly new in the sense that it does not require the equality between
all diffusion rates ($\mathbf{D}=\mathbf{I}$), which was required in \cite{Girardin_2017,Griette_Raoul}. The convergence results
for two-component systems presented in \cite{Cantrell_Cosner_Yu_2018} do not require such an assumption but use the boundedness 
of the domain to overcome this lack of structure; when extending these results to the unbounded setting, the equality $d_1=d_2$ 
is useful \cite[Appendix B]{Girardin_2018}.

\subsection{The nonlocal KPP equation}
The spatially homogeneous system 
\begin{equation*}
    \dot{\mathbf{u}}=\mathbf{M}\mathbf{u}+\mathbf{u}\circ(\mathbf{1}-\bm{\phi}\star\mathbf{u})
\end{equation*}
is, in a way, a discretized version of the nonlocal Fisher--KPP equation:
\begin{equation*}
    \partial_t u = \partial_{xx} u + u(1-\phi\star u).
\end{equation*}

This nonlocal equation has attracted a lot of attention in the last few years. The existing literature 
(\textit{e.g.}, \cite[and references therein]{Berestycki_Nadin_Perthame_Ryzhik,Alfaro_Coville_2012,Faye_Holzer_14,Bouin_Henderson_Ryzhik_2019})
develops new techniques to overcome the default of comparison principle and these techniques proved to be fruitful 
when applied to non-cooperative KPP systems \cite{Griette_Raoul,Girardin_2016_2}. 
In the present paper we will once again import such a technique from \cite{Berestycki_Nadin_Perthame_Ryzhik}.

\subsection{The nonlocal cane-toad equation}
The diffusive system \eqref{eq:RD_system} is, in a similar way, a discretized version of the nonlocal cane-toad equation:
\begin{equation*}
    \begin{cases}
	    \partial_t u = d(\theta)\partial_{xx} u+\alpha\partial_{\theta\theta}u + u(t,x,\theta)(1-\frac{1}{\overline{\theta}-\underline{\theta}}\int_{\theta=\underline{\theta}}^{\overline{\theta}} u(t,x,\theta')\textup{d}\theta'),\\
	\partial_\theta u(t,x,\underline{\theta})=\partial_\theta u(t,x,\overline{\theta})=0,	
    \end{cases}
\end{equation*}
where $u(t,x,\theta)$ is a population density structured with respect to a phenotypic trait $\theta\in[\underline{\theta},\overline{\theta}]\subset[0,+\infty]$. This eco-evolutionary model has also attracted attention recently
(\textit{e.g.}, \cite{Bouin_Calvez_2,Bouin_Henderso,Bouin_Calvez_2014,Bouin_Henderso-1,Calvez_Henders,Turanova,Alfaro_Coville_Raoul,Arnold_Desvill,Berestycki_Jin_Silvestre,Griette_2017}),
especially due to an acceleration phenomenon when $d(\theta)=\theta$ and $\overline{\theta}=+\infty$ but also because, 
just like the nonlocal KPP equation, it does not satisfy the comparison principle and requires new techniques.

It turns out that the similarity between our system and this equation is so strong that our proof can be readily adapted
and our result extends to this continuous-trait model (see Theorem \ref{thm:main_continuous} and its two corollaries).

\subsection{More general reaction terms}
In the system \eqref{eq:RD_system}, the reaction term has the form 
$\left(\mathbf{I}+\mathbf{M}\right)\mathbf{u}-\mathbf{u}\circ\left( \mathbf{C}\mathbf{u} \right)$.
It is natural to try to extend the results to reaction terms of the form 
$\left(\diag\left( \mathbf{r} \right)+\mathbf{M}\right)\mathbf{u}-\mathbf{u}\circ\left( \mathbf{C}\mathbf{u} \right)$,
where $\diag\left( \mathbf{r} \right)+\mathbf{M}$ has a positive Perron--Frobenius eigenvalue, or
$\left(\diag\left( \mathbf{r} \right)+\mathbf{M}\right)\mathbf{u}-\left(\diag\left( \mathbf{r} \right)\mathbf{u}\right)\circ\left( \mathbf{C}\mathbf{u} \right)$,
where $\mathbf{r}$ is positive.
However our proof does not easily extend to such cases. The Liouville-type result for such cases
remains as an open problem.

\subsection{The Cauchy problem}
It would be natural to try to prove that, with the same assumptions 
\ref{hyp:line-sum-symmetry}--\ref{hyp:stable} or 
\ref{hyp:cont_line-sum-symmetry}--\ref{hyp:cont_stable}, the solutions of the parabolic Cauchy
problem converge locally uniformly to $\mathbf{1}$. 
In fact, what can be proved, using the same test function 
$\left( (u_i-1)/u_i \right)_{i\in[N]}$ and the same calculations, is the locally uniform 
convergence to $\mathbf{1}$ under the additional assumption $\mathbf{D}=\mathbf{I}$. It is indeed
well-known that such an assumption makes it possible to use convex Lyapunov functions
for parabolic systems \cite{Weinberger_1975}~; in the present case, the convex Lyapunov function 
is $\mathbf{u}-\ln \mathbf{u}$ (where, obviously, the notation $\ln \mathbf{u}$ stands for
$\left( \ln u_i \right)_{i\in[N]}$). See also \cite[Section 7]{Ducrot_Giletti_Matano_2} for a 
similar argument.

The locally uniform convergence when $\mathbf{D}\neq \mathbf{I}$ remains as an open problem.

\section{Proof of Theorem \ref{thm:main}}

Our strategy is to mimic the proof of \cite[Theorem 4.1]{Berestycki_Nadin_Perthame_Ryzhik}, which uses the test function
$(p-1)/p$. More precisely, we rely upon 
\begin{equation}\label{eq:nonnegativity_1}
    \sum_{i=1}^N \frac{(\vect{M}\vect{p})_i}{p_i}\geq 0\quad\text{with equality iff }\mathbf{p}\in\vspan\left( \mathbf{1} \right),
\end{equation}
and upon
\begin{equation}\label{eq:nonnegativity_2}
    \sum_{i=1}^N (p_i-1)(\mathbf{C}(\mathbf{p}-\mathbf{1}))_i\geq 0.
\end{equation}
The inequality \eqref{eq:nonnegativity_1} is a standard property of irreducible line-sum-symmetric matrices (recalled in the 
forthcoming Lemma \ref{lem:line-sum-symmetry}); 
the inequality \eqref{eq:nonnegativity_2} is a direct consequence of (a generalized version of) Plancherel's theorem: 
\begin{align*}
    \sum_{i=1}^N q_i(\mathbf{C}\mathbf{q})_i & = \mathsf{Re}\left(\sum_{i=1}^N q_i(\mathbf{C}\mathbf{q})_i\right) \\
    & = \mathsf{Re}\left(\langle \vect{q}, \mathbf{C}\vect{q}\rangle\right)\\
    & = \mathsf{Re}\left(\langle \vect{q}, \overline{\mathbf{U}}^{\textup{T}}\mathbf{U}\mathbf{C}\overline{\mathbf{U}}^{\textup{T}}\mathbf{U}\vect{q}\rangle\right) \\
    & = \mathsf{Re}\left(\langle \mathbf{U}\vect{q}, \mathbf{U}\mathbf{C}\overline{\mathbf{U}}^{\textup{T}}\mathbf{U}\vect{q}\rangle\right) \\
    & \geq \min_{\lambda\in\operatorname{sp}\left( \mathbf{C} \right)}\left( \mathsf{Re}(\lambda) \right)\sum_{i=1}^N \left|\left(\mathbf{U}\mathbf{q}\right)_i\right|^2,
\end{align*}
where $\langle\cdot,\cdot\rangle$ is the canonical (Hermitian) scalar product on $\mathbb{C}^N$.

\begin{lem}[Classification of line-sum-symmetric matrices {\cite[Corollary 3]{Eaves_1985}}]\label{lem:line-sum-symmetry}
    Let $\mathbf{A}\in\mathscr{M}_{N,N}\left( \mathbb{R} \right)$ be a nonnegative matrix. Then $\mathbf{A}$ is line-sum-symmetric
    if and only if
    \begin{equation*}
	\sum_{i,j\in[N]}\frac{a_{i,j}u_j}{u_i}\geq \sum_{i,j\in[N]} a_{i,j}\quad\text{for all }\mathbf{u}\in\left( 0,+\infty \right)^N.
    \end{equation*}

    Furthermore, if $\mathbf{A}$ is irreducible and line-sum-symmetric, equality above holds if and only if 
    $\mathbf{u}\in\vspan\left( \mathbf{1} \right)$.
\end{lem}

The inequality \eqref{eq:nonnegativity_1} follows then from Lemma \ref{lem:line-sum-symmetry} applied to the nonnegative,
line-sum-symmetric and irreducible matrix $\mathbf{A}=\mathbf{M}-\min_{i\in[N]}(m_{i,i})\mathbf{I}$.
	
\begin{lem}[Uniqueness of the nonzero constant solution]\label{lem:uniqueness_constant}
    The unique nonnegative nonzero solution of $\mathbf{M}\mathbf{u}+\mathbf{u}=\left( \mathbf{C}\mathbf{u} \right)\circ\mathbf{u}$
    is $\mathbf{1}$.
\end{lem}
\begin{proof}
    Let $\mathbf{u}$ be any nonnegative nonzero solution. Recall that $\mathbf{u}$ is in fact positive. Denoting
    $\mathbf{u}^{\circ -1}=\left( 1/u_i \right)_{i\in[N]}$ and taking the scalar product
    \begin{equation*}
	    \langle -\mathbf{M}\mathbf{u}-\mathbf{u}\circ\left( \mathbf{1}-\mathbf{C}\mathbf{u} \right),\mathbf{u}^{\circ -1}\circ\left( \mathbf{u}-\mathbf{1} \right)\rangle=\mathbf{0},
    \end{equation*}
    we get 
    \begin{align*}
	    -\langle \mathbf{M}\mathbf{u},\mathbf{u}^{\circ -1}\rangle=\langle\mathbf{C}(\mathbf{u}-\mathbf{1}), \mathbf{u}-\mathbf{1}\rangle.
    \end{align*}
    On one hand, by \eqref{eq:nonnegativity_1}, the left-hand side is nonpositive. On the other hand, 
    by \eqref{eq:nonnegativity_2}, the right-hand side is nonnegative. Therefore both sides are zero. 
    Using now the case of equality in \eqref{eq:nonnegativity_1}, we deduce $\mathbf{u}\in\vspan\left( \mathbf{1} \right)$. 
	We deduce subsequently from the right-hand side that $\mathbf{u}=\mathbf{1}$.
\end{proof}

We are now in a position to prove Theorem \ref{thm:main}.
\begin{proof}[Proof of Theorem \ref{thm:main}]
    Let $(\mathbf{p},c)$ be a bounded solution of the system \eqref{eq:TW} satisfying $\inf_{\mathbb{R}}p_i>0$ for any $i\in[N]$.\medskip

    \noindent\textbf{Step 1:} We show that 
    \begin{equation*}
	    \lim_{\xi\to\pm\infty}\mathbf{p}(\xi)=\mathbf{1}.
    \end{equation*}
    
	At any $\xi\in\mathbb{R}$, denoting
    $\mathbf{p}^{\circ -1}(\xi)=\left( 1/p_i(\xi) \right)_{i\in[N]}$ and taking the scalar product (in $\mathbb{R}^N$)
    \begin{equation*}
	\langle-\mathbf{D}\mathbf{p}''-c\mathbf{p}'-\mathbf{M}\mathbf{p}-\mathbf{p}\circ\left( \mathbf{1}-\mathbf{C}\mathbf{p} \right),\mathbf{p}^{\circ -1}\circ\left( \mathbf{p}-\mathbf{1} \right)\rangle=\mathbf{0},
    \end{equation*}
    we get
    \begin{equation*}
	\sum_{i=1}^N\left[ -(d_ip_i''+cp_i')\left( \frac{p_i-1}{p_i} \right) \right] =
	-\sum_{i=1}^N\frac{\left( \mathbf{M}\mathbf{p} \right)_i}{p_i}
	-\sum_{i=1}^N\left( p_i-1 \right)\left( \mathbf{C}\left(\mathbf{p}-\mathbf{1}\right) \right)_i.
    \end{equation*}

    By \eqref{eq:nonnegativity_1} and \eqref{eq:nonnegativity_2}, the right-hand side is nonpositive and therefore 
	\begin{equation*}
	\sum_{i=1}^N\left[ -(d_ip_i''+cp_i')\left( \frac{p_i-1}{p_i} \right) \right] \leq 0.
    \end{equation*}

    Since this holds true at any $\xi\in\mathbb{R}$, we fix $R>0$ and integrate by parts in $[-R,R]$. We get, as in
    \cite[Proof of Lemma 4.1]{Berestycki_Nadin_Perthame_Ryzhik},
    \begin{equation}\label{eq:estIPP}
	\sum_{i=1}^N d_i\int_{-R}^R \left(\frac{p_i'}{p_i}\right)^2\leq\sum_{i=1}^N\left[ d_i\frac{p_i'\left( p_i-1 \right)}{p_i}+c\ln(p_i)-cp_i \right]_{-R}^R,
    \end{equation}
	By the classical elliptic estimates, $|p_i'(\pm R)|$ is bounded by 
	$\max_{i\in[N]}\sup_{\mathbb{R}}p_i$
	up to a multiplicative constant independent of $R$. 
	Recalling that $p_i$ is uniformly bounded from below by $\min_{i\in[N]}\inf_{\mathbb R}p_i>0$, the right-hand side is bounded by a constant independent of $R$. We deduce that $p_i'\in L^2(\mathbb R) $ for all $i\in [N]$.

	Let now $(\xi_n)_{n\in\mathbb{N}}$ be any sequence such that $\xi_n\to -\infty$ and 
	define $\mathbf{p}^n:\xi\mapsto\mathbf{p}(\xi+\xi_n)$. We remark that, for all 
	$i\in [N]$ and all $\zeta\in\mathbb{R}$, we have
	\begin{equation*}
		\int_{-\infty}^{\zeta}[(p_i^n)']^2(\xi)\dd \xi = \int_{-\infty}^{\zeta+\xi_n}(p_i')^2(\xi) \dd \xi\xrightarrow[n\to+\infty]{}0 \text{ for all }i\in[N], 
	\end{equation*}
	and therefore $(\mathbf{p}^n)'$ converges to $\mathbf{0}$ locally uniformly in $L^2$. 
	Next, using the classical elliptic estimates, we extract from 
	$(\mathbf{p}^n)_{n\in\mathbb{N}}$ a subsequence which converges in 
	$\mathscr{C}^2_{\textup{loc}}$ to a limit $\mathbf{p}^\infty\in \mathscr{C}^2(\mathbb R)$. 
	Note that $\mathbf{p}^\infty$ is still a solution to \eqref{eq:TW}. Since 
	$(\mathbf{p}^n)'\to 0$ in $L^2_{\textup{loc}}$, we conclude that $\mathbf{p}^\infty$ 
	has to be a constant function of the variable $\xi$, \textit{i.e.} a constant solution of 
	$\mathbf{M}\mathbf{p}+\mathbf{p}=\left( \mathbf{C}\mathbf{p} \right)\circ \mathbf{p}$. 
	By Lemma \ref{lem:uniqueness_constant}, $\mathbf{p}^\infty = \mathbf{1}$ identically. 
	
	Since the sequence $(\xi_n)_{n\in\mathbb{N}}$ is arbitrary, we have shown that 
	\begin{equation*}
		\lim_{\xi\to -\infty}\mathbf{p}(\xi)=\mathbf{1}.
	\end{equation*}
	The limit at $+\infty$ can be established by a similar argument.\medskip

	\noindent \textbf{Step 2:} We show that $\mathbf{p}$ is identically equal to $\textbf{1}$.
	Since $\mathbf{p}$ converges to $\mathbf{1}$ on both sides of the real line, the brackets on the right-hand side of \eqref{eq:estIPP} converge to $0$ as $R\to +\infty$. Therefore, 
	\begin{align*}
		0\leq \sum_{i=1}^N\int_{-\infty}^{+\infty}(p_i')^2&=\lim_{R\to+\infty}\sum_{i=1}^N\int_{-R}^{R}(p_i')^2 \\
		& \leq \lim_{R\to+\infty}\dfrac{\displaystyle\max_{i\in[N]}\sup_{\mathbb{R}} p_i^2}{\displaystyle\min_{i\in[N]} d_i}\sum_{i=1}^N d_i\int_{-R}^R \left(\frac{p_i'}{p_i}\right)^2 \\ 
		&\leq C\lim_{R\to+\infty}\sum_{i=1}^N\left[ d_i\frac{p_i'\left( p_i-1 \right)}{p_i}+c\ln(p_i)-cp_i \right]_{-R}^R = 0, 
	\end{align*}
	where $C$ is a constant independent of $R$. We conclude that $\mathbf{p}$ has to be a constant function of $\xi$, and the only possibility is $\mathbf{p}= \mathbf{1}$. 
\end{proof}

\begin{rem}\label{rem:sharpness_of_A1}
    From the proofs of Lemma \ref{lem:uniqueness_constant} and Theorem \ref{thm:main} above, 
    it is clear that the estimate $\langle \mathbf{u}^{\circ -1},\mathbf{M}\mathbf{u}
\rangle \geq 0$, together with the equality case, is crucial. Since this estimate fails if $\mathbf{M}$ is not line-sum-symmetric
(by Lemma \ref{lem:line-sum-symmetry}), \ref{hyp:line-sum-symmetry} is sharp regarding the proof presented here.
Note that \cite[Theorem 2]{Eaves_1985} proved that every nonnegative irreducible matrix $\mathbf{A}$ has a line-sum-symmetric 
similarity-scaling $\diag\left( \mathbf{x} \right)\mathbf{A}\diag\left( \mathbf{x} \right)^{-1}$, where $\mathbf{x}$ is a positive vector,
but it seems to us that this property cannot be used to generalize the above proof to non-line-sum-symmetric matrices $\mathbf{M}$.
\end{rem}

Finally we recall briefly the arguments leading to the proof of Corollary \ref{cor:unique} and \ref{cor:limit_TW}.
\begin{proof}[Sketch of the proof of Corollary \ref{cor:unique}]
	Let $\mathbf{p}$ be a bounded steady state solution of \eqref{eq:RD_system}, \textit{i.e.} a 
	traveling wave with speed $c=0$. 
	It is known from \cite[Theorem 1.3 \textit{(ii)}]{Girardin_2016_2} that, if $\mathbf{p}$ is nonnegative and nonzero, 
	then $\mathbf{p}$ is bounded uniformly away from $\mathbf{0}$. Theorem \ref{thm:main} concludes.
\end{proof}
\begin{proof}[Sketch of the proof of Corollary \ref{cor:limit_TW}]
	Let $(\mathbf{p}, c)$ be a bounded traveling wave solution of 
	\eqref{eq:RD_system} satisfying the boundary conditions \eqref{eq:asymptotic_conditions}. 
	By \cite[Theorem 1.5 \textit{(iii)}]{Girardin_2016_2}, condition \eqref{eq:asymptotic_conditions} near $-\infty$ immediately transfers to 
	\begin{equation*}
		\min_{i\in[N]}\liminf_{\xi\to-\infty}p_i(\xi)>0,
	\end{equation*}
	therefore Theorem \ref{thm:main} can be applied to any local uniform limit of a converging sequence $\mathbf{p}(\xi+\xi_n)$ for some $\xi_n\to -\infty$. Since the limit is uniquely identified, the claim is proved.
\end{proof}

\section{Proof of Theorem \ref{thm:main_continuous}}

We follow the same steps as for the proof of Theorem \ref{thm:main}, but have to adapt each argument in the correct functional setting.

\begin{lem}[Positivity of $K$]\label{lem:cont_pos}
	Assume \ref{hyp:cont_normal} and \ref{hyp:cont_stable}. Then for all nonzero $u\in L^2(\Omega)$,
	\begin{equation*}
		\langle K[u], u\rangle_{L^2(\Omega)}=\int_{\Omega^2} u(y)k(y, z)u(z)\dd y\dd z\geq0.
	\end{equation*}
\end{lem}
\begin{proof}
	To prove the result, we take advantage of the spectral decomposition of $K$ considered as an operator acting on the complex Hilbert space $L^2_{\mathbb C}(\Omega)$ equipped with the canonical hermitian product $\langle f, g\rangle_{L^2_{\mathbb C}}=\int_{\Omega}f\overline{g}$. Clearly $K$ is still normal when considered as an operator on $L^2_{\mathbb C}(\Omega)$. Moreover, by  Lemma \ref{lem:compactness}, $K$ is compact (and the compactness classically transfers to the complex extension of $K$). Since $L^2_{\mathbb C}(\Omega)$ is separable, by the spectral decomposition theorem (see e.g. \cite[Proposition 11.36 p.369]{Brezis_2011}), there exists a Hilbert basis of $L^2_{\mathbb C}(\Omega)$ composed of eigenvectors of $K$. Let us denote $(e_n)_{n\in\mathbb N}$ such a Hilbert basis and $(\lambda_n)_{n\in\mathbb N}$ the corresponding sequence of eigenvalues. This decomposition yields
	\begin{equation*}
		\langle K[u], u\rangle_{L^2_\mathbb C} = \sum_{n=0}^{+\infty}\lambda_n|\langle u, e_n\rangle_{L^2_{\mathbb C}}|^2, 
	\end{equation*}
	but since $\langle K[u], u\rangle_{L^2_{\mathbb C}} $ is real,
	\begin{equation*}
		\langle  K[u], u\rangle_{L^2_\mathbb C} = \Real\left(\sum_{n=0}^{+\infty}\lambda_n|\langle u, e_n\rangle_{L^2_{\mathbb C}}|^2\right)=\sum_{n=0}^{+\infty}\Real(\lambda_n)|\langle u, e_n\rangle_{L^2_{\mathbb C}}|^2\geq 0.\qedhere
	\end{equation*}
\end{proof}
\begin{lem}[Characterization of continuous line-sum-symmetric operators {\cite[Theorem 4]{Cantrell_Cosner_Lou_Ryan_2012}}]\label{lem:cont_M}
    Let $a\in\mathscr{C}\left(\Omega\times\Omega,[0,+\infty)\right)$ be Riemann integrable. Then the following two properties
    are equivalent:
    \begin{enumerate}
	\item $\int_\Omega a(x,y)\textup{d}y=\int_\Omega a(y,x)\textup{d}y$ for all $x\in\Omega$;
	\item $\int_{\Omega\times\Omega} \frac{a(x,y)u(y)}{u(x)}\textup{d}y\textup{d}x\geq\int_{\Omega\times\Omega} a(x,y)\textup{d}y\textup{d}x$ 
	    for all $u\in\mathscr{C}\left(\overline{\Omega},\left( 0,+\infty \right)\right)$.
    \end{enumerate}
\end{lem}

We point out that the equality case of the second property is not presented in the above lemma but was studied in 
\cite[Theorem 5]{Cantrell_Cosner_Lou_Ryan_2012}
under the irreducibility-type assumption that $(x,y)\mapsto a(x,y)+a(y,x)$ does not vanish. Here, we need in any case to include
in the mutation operator a nontrivial divergence part ($\sigma>0$), and this suffices for the irreducibility-type properties we need,
so that we do not make any irreducibility-type assumption on the nonlocal part.

\begin{lem}[Uniqueness of the constant solution]\label{lem:cont_unique}
	Assume \ref{hyp:cont_line-sum-symmetry}, \ref{hyp:cont_normal} and \ref{hyp:cont_stable}. The constant 1 is the unique nonnegative nonzero classical solution to the equation
	\begin{equation}\label{eq:stat}
		\nabla_y\cdot(\sigma(y)\nabla_y p)+M[p](y)+p(y)\left( 1-K[p](y) \right)=0,
	\end{equation}
	supplemented with homogeneous Neumann boundary conditions on $\partial\Omega$.
\end{lem}
\begin{proof}
	We first remark that, by a direct application of the strong maximum principle and Hopf's lemma, the
	fact that $p$ is nonzero can be reinforced as $p(y)>0$ on $\overline{\Omega}$.  Since moreover $p$ is continuous on $\overline{\Omega}$, $p$ is bounded from below. In particular, the test function $\frac{p(y)-1}{p(y)}$ is well-defined and (at least) in $\mathscr{C}^1(\overline{\Omega})$. 

	As in the discrete case, we multiply \eqref{eq:stat} by  $\frac{p(y)-1}{p(y)} $ and integrate over $\Omega$. Integrating by parts the gradient term, we get:
	\begin{align*}
		0= & -\int_{\Omega}\sigma(y)\nabla_y p(y)\nabla_y\left(1-\frac{1}{p(y)}\right)\dd y+\int_{\Omega\times\Omega} m(y,z)(p(z)-p(y))\dd z\dd y\\
		& -\int_{\Omega\times\Omega} m(y,z)(p(z)-p(y))\frac{1}{p(y)}\dd z\dd y +\int_{\Omega}\left( 1-K[p](y) \right)(p(y)-1)\dd y.
	\end{align*}

	Let us show that each of those terms is nonpositive. We first remark that 
	\begin{equation*}
		-\int_{\Omega}\sigma(y)\nabla_y p(y)\nabla_y\left(1-\frac{1}{p(y)}\right)\dd y=-\int_{\Omega}\sigma(y)\frac{|\nabla p|^2}{p(y)^2}\dd y\leq 0,
	\end{equation*}
	\begin{align*}
	    \int_{\Omega\times\Omega} m(y,z)(p(z)-p(y))\dd z\dd y & 
	    =\int_\Omega \left( \int_\Omega m(y,z)p(z)\dd z -\int_\Omega m(y,z)\dd z p(y)\right)\dd y \\
	    & =\int_\Omega \left( \int_\Omega m(y,z)p(z)\dd z -\int_\Omega m(z,y)\dd z p(y)\right)\dd y \\
	    & =\int_{\Omega\times\Omega} m(y,z)p(z)\dd z\dd y -\int_{\Omega\times\Omega} m(z,y) p(y)\dd z\dd y \\
	    & = 0.
	\end{align*}
	Next, by Lemma \ref{lem:cont_M},
	\begin{equation*}
	    \int_{\Omega\times\Omega} m(y,z)(p(z)-p(y))\frac{1}{p(y)}\dd y\dd z = \int_{\Omega\times\Omega}\frac{m(y,z)p(z)}{p(y)}\dd y \dd z - \int_{\Omega\times\Omega}m\geq 0.
	\end{equation*}
	Finally, since $K[1]=1$, we have $1-K[p]=K[1-p]$ and thus, by Lemma \ref{lem:cont_pos},
	\begin{equation*}
		\int_{\Omega}\left( 1-K[p](y) \right)(p(y)-1)\dd y=-\int_{\Omega}K[1-p](y)(1-p(y))\dd y\leq 0.
	\end{equation*}

	Therefore each of those four terms is in fact equal to 0.
	From $\int_{\Omega}\sigma(y)\frac{|\nabla p(y)|^2}{p(y)^2}\dd y=0$ we deduce that $p(y)$ is a constant on $\overline{\Omega}$. Since then
	\begin{equation*}
		0=\int_{\Omega}K[1-p](1-p)\dd y=(1-p)^2\int_{\Omega^2}k(y,z)\dd y\dd z
	\end{equation*}
	and $\int_{\Omega^2}k(y,z)\dd y\dd z>0$, we conclude that $p=1$.
\end{proof}
We are now in a position to prove Theorem \ref{thm:main_continuous}. 
\begin{proof}[Proof of Theorem \ref{thm:main_continuous}]
	As in the discrete case (proof of Theorem \ref{thm:main}), we multiply the equation \eqref{eq:TW_continuous} by the test 
	function $\frac{p(\xi,y)-1}{p(\xi,y)} $ and integrate on the cylinder $\Omega_R=[-R, R]\times\Omega$ for some $R>0$. With
	the exact same computations as in the proof of Lemma \ref{lem:cont_unique}, we get
	\begin{equation*}
		-\int_{\Omega_R}(d(y)\partial_{\xi\xi}p(\xi, y)+c\partial_\xi p(\xi, y))\frac{p(\xi, y)-1}{p(\xi, y)}\dd\xi\dd y \leq 0.
	\end{equation*}
    After integrations by parts in the $\xi$ variable, we find
	\begin{multline}\label{eq:gradest-cylinder}
		\int_{\Omega_R}d(y)\frac{|\partial_\xi p(\xi, y)|^2}{p(\xi, y)^2}\dd \xi\dd y\\
		\leq \left[\int_\Omega d(y)\frac{\partial_\xi p(\xi, y)(p(\xi, y)-1)}{p(\xi, y)}+c\left(\ln(p(\xi, y))-p(\xi, y)\right)\dd y\right]_{-R}^R
	\end{multline}
	where, by the classical elliptic estimates, $|\partial_\xi p(\pm R, y)|$ is controlled from above by $\sup_{(\xi,y)\in\mathbb{R}\times\Omega}p(\xi, y)$, independently of $R$. Taking the limit $R\to +\infty$, we see that $\partial_{\xi}p(\xi, y)\in L^2(\mathbb R\times\Omega)$. Using elliptic regularity, a translation argument (which is similar to the one developed 
	in the proof of Theorem \ref{thm:main}) and Lemma \ref{lem:cont_unique}, we conclude that 
	\begin{equation*}
		\lim_{\xi\to\pm\infty} \sup_{y\in\Omega} |p(\xi, y)-1|=0.
	\end{equation*}
	Going back to \eqref{eq:gradest-cylinder}, we easily see that the right-hand side converges to zero as $R\to +\infty$ and therefore 
	\begin{equation*}
		\int_{\mathbb R\times\Omega}d(y)\frac{|\partial_\xi p(\xi, y)|^2}{p(\xi, y)^2}\dd \xi\dd y =0,
	\end{equation*}
	thus $p$ is constant in $\xi$. In view of the limiting conditions, we conclude 
	that in fact $p=1$ identically. 
\end{proof}

Corollary \ref{cor:unique_cont} is a direct application of Theorem \ref{thm:main_continuous}.
As for Corollary \ref{cor:TW_limit_cont}, it is proven by an argument similar to the one that yields the limit of the solution near $\pm \infty$ in the proof of Theorem \ref{thm:main_continuous}. Since it  is rather classical to adapt this argument for traveling waves, we omit the details. 

We end by a technical but necessary lemma.

\begin{lem}[Compactness of $K$]\label{lem:compactness}
	Assume \ref{hyp:cont_normal}. Then the operator $K:L^2(\Omega)\to L^2(\Omega)$ is compact. 
\end{lem}
\begin{proof}
    We aim at applying the Kolmogorov--Riesz--Fr\'{e}chet Theorem (see e.g. \cite[Theorem 4.26 p.111]{Brezis_2011}) to our operator $K$. We extend the function $k(y,z)$ to $\mathbb R^Q\times\mathbb R^Q$ by setting $k(y,z)=0$ for $y, z\not\in\Omega$. For $f\in L^2(\mathbb R^Q)$ we define:
	\begin{equation*}
		K[f](y)=\int_{\mathbb R^Q}k(y,z)f(z)\dd z.
	\end{equation*}
	Let $\varepsilon>0$ and $f\in L^2(\Omega)$, $\Vert f\Vert_{L^2(\Omega)}=1$ be given. 
	We extend $f$ to $L^2(\mathbb R^Q)$ by setting $f(z)=0$, $z\not\in{\Omega}$. 
	We also define, for all $h\in\mathbb{R}^Q$, the translation operator 
	$\tau_h: g \mapsto g(\bullet+h)$.
	We first remark that, for any $h\in\mathbb{R}^Q$,
	\begin{align*}
		\Vert \tau_hK[f]-K[f]\Vert_{L^2}^2&=\int_{\mathbb R^Q} \bigg(\int_{\mathbb R^Q}k(y+h, z)f(z)\dd z  
		 -\int_{\mathbb R^Q}k(y,z)f(z)\dd z\bigg)^2\dd y \\
		&=\int_{\mathbb R^Q} \bigg(\int_{\mathbb R^Q}(k(y+h, z) -k(y,z))f(z)\dd z\bigg)^2\dd y\\
		&\leq \int_{\mathbb R^Q}\int_{\mathbb R^Q} (k(y+h,z)-k(y, z))^2\dd z\dd y\Vert f\Vert_{L^2}\\
		&= \int_{\mathbb R^Q}\int_{\mathbb R^Q} (k(y+h,z)-k(y, z))^2\dd z\dd y,
	\end{align*}
	where we have used the classical Cauchy--Schwarz inequality in $L^2(\mathbb R^Q)$. 
	Therefore it only remains to control the $L^2$ norm of $k(y+h, \misvar)-k(y, \misvar)$ 
	when $h$ is small.
To this aim we fix $\delta_1>0$ be such that 
	\begin{equation*}
		|\{d(y,\partial\Omega)\leq \delta_1\}|\leq \frac{\varepsilon}{8\Vert k\Vert_{L^\infty(\Omega^2)}^2 |\Omega|},
	\end{equation*}
	where $d(\misvar, \partial\Omega)$ is the Euclidean distance between $y\in\mathbb R^Q$ and the set $\partial \Omega$ and $|\{d(y, \partial\Omega)\leq \delta_1\}|$ is the Lebesgue measure of the set of points $y\in\mathbb R^Q$ satisfying $d(y, \partial\Omega)\leq \delta_1$. Since $k$ is continuous on the compact set $\overline{\Omega}^2$, there exists $\delta_2>0$ such that $|k(y+h, z)-k(y,z)|\leq {\frac{\varepsilon}{\sqrt{2}|\Omega|}}$ if  $y, y+h, z\in\overline\Omega$ and $|h|\leq \delta_2$. 

	Therefore, if $|h|\leq \min(\delta_1, \delta_2)$, we have:
	\begin{align*}
		\Vert \tau_hK[f]-K[f]\Vert_{L^2}^2&= \int_{d(y, \partial\Omega)\leq \delta_1}\int_{\mathbb R^Q} (k(y+h,z)-k(y, z))^2\dd z\dd y \\
		&\quad+\int_{d(y, \partial\Omega)> \delta_1}\int_{\mathbb R^Q} (k(y+h,z)-k(y, z))^2\dd z\dd y\\
		&\leq 4|\Omega|\Vert k\Vert_{L^\infty} ^2|\{d(y,\partial\Omega)\leq \delta_1\}|+|\Omega|^2\frac{\varepsilon^2}{2|\Omega|^2}\\
		&\leq \varepsilon^2 .
	\end{align*}
	We conclude that  $K$ is indeed compact on $L^2(\Omega)$. 
\end{proof}

\section*{Acknowledgments}
The authors wish to thank Benoit Perthame for the attention he paid to this work and Chris Cosner for pointing out the theory of
line-sum-symmetric matrices that substantially improved the results.

\bibliographystyle{plain}
\bibliography{ref.bib}

\end{document}